\documentclass[a4paper,11pt]{amsart}
\usepackage{hyperref,latexsym}
\usepackage{enumerate}
\usepackage{graphicx}
\usepackage{color}
\usepackage{amsmath}

\theoremstyle{plain}
\newtheorem{theorem}{Theorem}[section]
\newtheorem{lemma}[theorem]{Lemma}
\newtheorem{corollary}[theorem]{Corollary}
\newtheorem{proposition}[theorem]{Proposition}
\theoremstyle{definition}

\theoremstyle{remark}

\newcommand{\abs}[1]{\left|#1\right|}

\begin{document}

\title[]
      {Effective rigidity away from the boundary for centrally-symmetric billiards}

\date{25 September 2022}
\author{Misha Bialy}
\address{School of Mathematical Sciences, Raymond and Beverly Sackler Faculty of Exact Sciences, Tel Aviv University,
Israel} 
\email{bialy@tauex.tau.ac.il}
\thanks{MB was partially supported by ISF grant 580/20 and DFG grant MA-2565/7-1  within the Middle East Collaboration Program.}

%\subjclass[2000]{} 
%\keywords{}

\begin{abstract} In this paper we study centrally symmetric Birkhoff billiard tables. We introduce a closed invariant set $\mathcal{M}_\mathcal{B}$ consisting of locally maximizing orbits of the billiard map lying inside the region $\mathcal{B}$ bounded by two invariant curves of $4$-periodic orbits.  We give an effective bound from above on the measure of this invariant set in terms of the isoperimetric defect of the curve. The equality case occurs if and only if the curve is a circle. 
\end{abstract}

\maketitle

%%%%%%%%%%%%%%%%%%%%%%%%%%%%%%%%%%%%%%%%%%%%%%%%%%%%%%%%%%%%%%%%%%%%%%%%%%

\section{Introduction}
In this paper we study Birkhoff billiards for centrally symmetric $C^2$-smooth convex curves in the plane. We introduce the set
$\mathcal{M}_{\mathcal B}$ lying in the region $\mathcal{B}$ between two invariant curves $\alpha, \bar{\alpha}$ in the phase space (see Figure \ref{B}). The set $\mathcal{M}_{\mathcal B}$, by definition,  consists of those orbits such that any finite sub-segment is locally maximizing, for the length functional $\mathcal{L}$ associated to the billiard table. We assume that $\alpha, \bar{\alpha}$ consist of $4$-periodic orbits of rotation numbers $1/4$ and $3/4$ respectively. It then follows that the set $\mathcal{M}_\mathcal{B}$ is a closed set which is invariant under the billiard map $T$.
\begin{figure}[h]\label{B}
	\centering
	\includegraphics[width=0.4\linewidth]{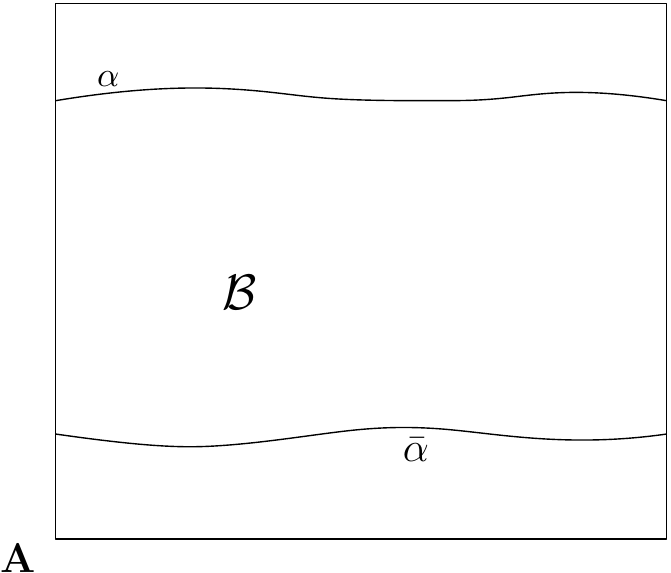}
	\caption{The region $\mathcal B$}
	%\label{main}
\end{figure}
Our goal in this paper is to get an upper bound on the measure of the set $\mathcal{M}_\mathcal B$ which is sharp, i.e. the case when  $\mathcal{M}_\mathcal B$ occupies the whole of $\mathcal{B}$ occurs if and only if the billiard table is circular. Thus we show that the measure of the complement set $\Delta_\mathcal B:=\mathcal B\setminus\mathcal M_\mathcal B$ can be estimated from below in terms of the isoperimetric defect of the billiard domain.

 These bounds are of obvious importance for classical dynamics (and probably also for quantum properties), because all "rotational" invariant curves, as well as
Aubry-Mather sets, are filled by orbits which are locally length maximizing  (we refer to the monographs \cite{Tab}\cite{bangert}\cite{siburg}\cite{KT} for background material).

Estimates of this type  were obtained previously in \cite{B-effective} \cite{Bialy-Tsodikovich}, as an effective version of the so called E.Hopf rigidity phenomenon for billiards. 

The estimate presented here is related to the recent progress in the 
Birkhoff conjecture \cite{BM-annals} for centrally symmetric billiard tables. Similarly to \cite{BM-annals} we consider here the class of $C^2$-billiard tables having invariant curve consisting of $4$-periodic orbits and use its properties. We refer here to papers \cite{AB}\cite{AKS}\cite{K-S}\cite{KS2}\cite{G} for other powerful recent approaches. However, the main novelty of the present paper is that the region $\mathcal B$ lies away from the boundary of the phase cylinder.

It is an open question how to remove the restriction of central symmetry
of the billiard table. It is also interesting if effective bounds can be found for a region between two arbitrary invariant curves in the phase space. We now turn to the needed background and the formulation of the main result.

Let $\gamma $ be a $C^2$-smooth simple closed convex curve of positive curvature in $\mathbf R^2$. 
We fix the counterclockwise orientation on $\gamma$. We shall use the arclength parametrization $s$ as well as the parametrization by 
the angle $\psi$ formed by the outer unit normal $n$ to $\gamma$ with a fixed direction. These two parametrizations are related by $d\psi=k(s) ds$, where $k(s)$
is the curvature at the point $\gamma(s)$.

The natural phase space of the Birkhoff billiard inside $\gamma$ is the space $\mathbf A$ of all oriented lines that  intersect $\gamma$.
This space is topologically  a cylinder and we shall refer to it as the phase cylinder of $T$.
The billiard map $T$ acts on $\mathbf A$ by the reflection law in $\gamma$.
The phase cylinder carries a natural symplectic structure that can be described as follows.

Each oriented line is identified with the pair $(\cos\delta,s), \  \delta\in(0,\pi)$, where $\gamma(s)$ is the incoming point and $\delta$ is the angle between the line and the tangent $\gamma'(s)$. In these coordinates the symplectic form is 
$d\lambda$, where $\lambda=\cos\delta\, ds$, where $\cos\delta$ plays the role of momentum variable.
We shall denote by $\mu$ the corresponding invariant measure on the phase space $\mathbf{A}$.
The billiard map $T$ is a symplectic map and the chord length $\quad L(s,s_1)=|\gamma(s)-\gamma(s_1)|$ is a generating function of $T$ (see Figure \ref{fig:L}) 
Namely, $$ T^*\lambda-\lambda=\cos\delta_1\,ds_1-\cos\delta\, ds=dL.
$$

\begin{figure}[h]
	\centering
	\includegraphics[width=0.6\textwidth]{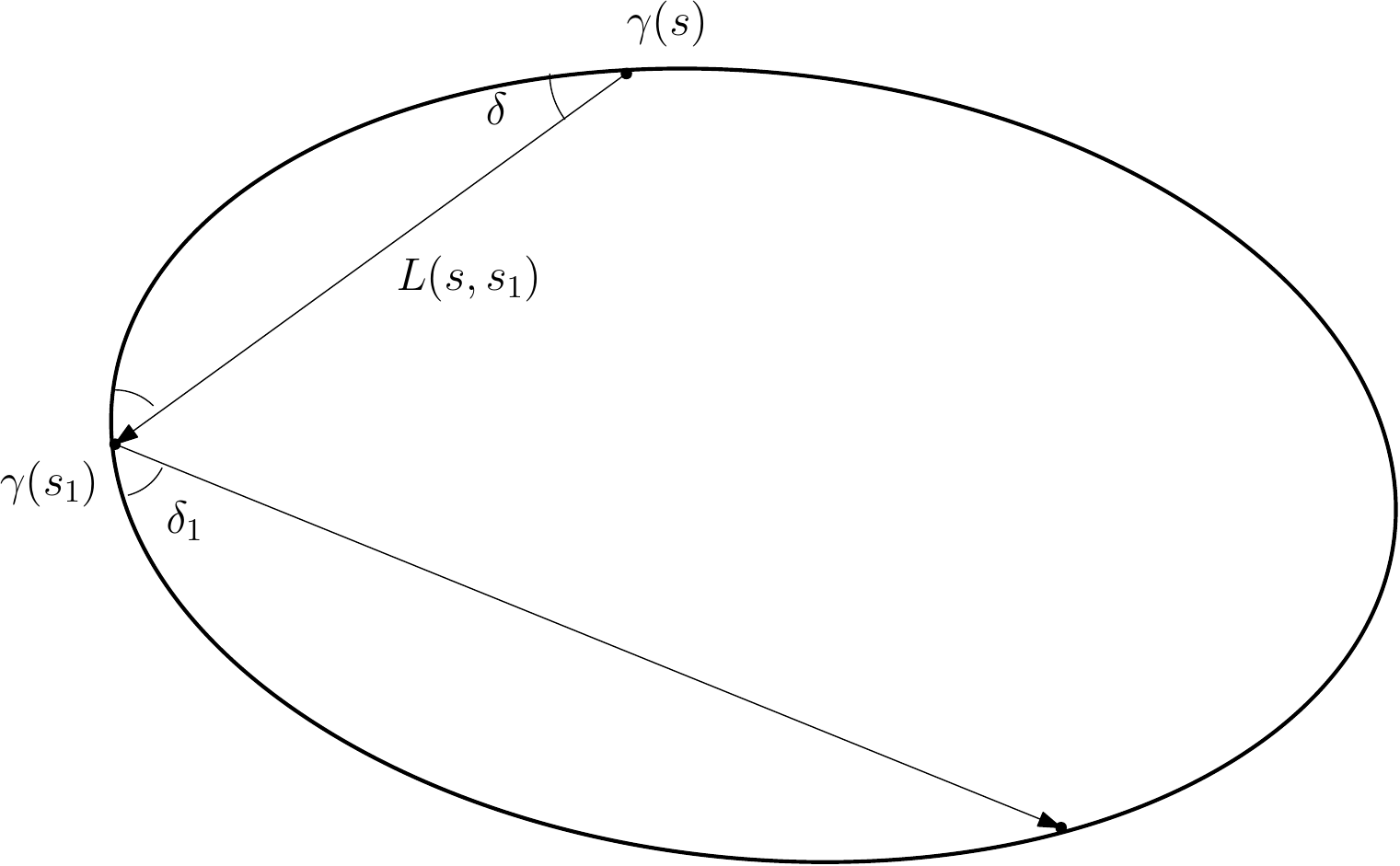}
	\caption{Generating function $L$ corresponding to the 1-form $\lambda$.}
	\label{fig:L}
\end{figure}
 Moreover, one can check that $T$ satisfies the twist condition:
\begin{equation}\label{eq:twist}
	L_{12} (s,s_1) >0,
\end{equation}
meaning that $T$ is a {\it negative} twist symplectic map (here and below we use subindex $1$ and/or $2$ for the partial derivative with respect to the first or the second argument respectively).
 
%%%%%%%%%%%%%%%%%%%%%%%%%%
{\bf Remark.} Let me remark that traditionally the generating function is the negative of ours, i.e.  the negative chord length. However we prefer, for convenience, sign $+$ for the generating function and hence the twist condition (\ref{eq:twist}) for the billiard map. Consequently we deal with maximizing (and not minimizing) orbits. 
%%%%%%%%%%%%%%%%%%%%%%%%%%%%%%%

For the generating function $L$ we can naturally define the variational principle as follows. For the configuration sequence $\{s_n\}$ we associate the formal sum \[\mathcal L \{s_n\}=\sum_n L(s_n,s_{n+1}).\] Configurations $\{s_n\}$, corresponding to a billiard trajectories  are critical points of the functional $\mathcal L$.

We shall consider {\it locally maximizing} configurations, that is, those configurations which give {\it local maximum}
for the functional between any two end-points.
We shall call such configurations, \textit{m-configurations}, and the corresponding orbits on the phase cylinder $\mathbf{A}$, \textit{m-orbits}.
We denote by $\mathcal M\subset\mathbf{A}$ the set swept by all m-orbits corresponding to the variational principle for the generating function $L$. We shall also use the following notations
$$
\mathcal M_\mathcal B:=\mathcal M \cap \mathcal B,\quad \Delta_\mathcal B:=\mathcal B\setminus\mathcal M_\mathcal B
$$
Let $\gamma\subset\mathbf{R}^2$ be a $C^2$-smooth, centrally symmetric, convex closed curve of positive curvature. We shall assume that the billiard map corresponding to $\gamma$ has a rotational (=winding once around the cylinder and simple) invariant curve $\alpha\subset \mathbf{A}$ consisting of $4$-periodic orbits.
We shall denote by $\bar\alpha$ the corresponding invariant curve of 
rotation number $\frac{3}{4}$. This curve consists of the same billiard trajectories but with the reversed orientation of the lines.
Our main result is the following:
\begin{theorem}\label{main1}
	Suppose that the billiard ball map $T$ of $\gamma$  has a continuous rotational  invariant curve $\alpha\subset \mathbf A$  of rotation number $1/4$, consisting of $4$-periodic orbits. Let $\bar\alpha$ be the corresponding invariant curve of 
	rotation number $\frac{3}{4}$.
	Let  $\mathcal B\subset \mathbf A$ be the domain between the curves $\alpha$ and $\bar\alpha$ (see Figure \ref{B}).
	Then the following estimate holds:
	\begin{equation}
\frac{3\beta}{16}(P^2-4\pi A)\leq \mu(\Delta_\mathcal B),
	\end{equation} 
where $P$,$A$ denote the perimeter and the area of $\gamma$, and $\beta>0$ is the minimal curvature of $\gamma$.
	\end{theorem}
Sharp estimates for $\mathcal M$ were obtained first in \cite {B-effective} and then in \cite{Bialy-Tsodikovich} as quantitative version of the so called E.Hopf rigidity phenomenon for billiards discovered in \cite{B0} and then \cite{B1}\cite{W}. In the paper \cite{Bialy-Tsodikovich} the region between the invariant curve $\alpha$ and the boundary of the phase cylinder was considered, while in the present paper the  region $\mathcal B$ lies away from the boundary. The significance of the invariant curve of 4-$periodic$ orbits was first understood in \cite{BM-annals}, and we shall remind the properties of this curve here and use them below.

Here are some useful corollaries of Theorem \ref{main1}:
\begin{corollary}\label{circle}
	Set $\mathcal M_\mathcal B$ of locally maximizing orbits occupies the whole region $\mathcal B$ if and only $\gamma$ is a circle.
\end{corollary}
In fact one can reformulate Corollary \ref{circle} in a dynamical way:
\begin{corollary}\label{main2}
Suppose that the  restriction of billiard map $T$ to $\mathcal B$ has an invariant measurable field of non-vertical oriented lines, with the orientation chosen on the lines coherently by the condition $ds>0$. Then $\gamma $ is a circle.
\end{corollary}

This is especially useful in establishing the following geometric fact.
\begin{corollary}\label{cor:vertical}
	If $\gamma$ is not a circle, then there always exist a point $x\in\mathcal B$ and a vertical tangent vector $v\in T_x\mathcal B$ such that for some positive integer $n$, the vector $DT^n(v)$ is again vertical (this  exactly means that the points $x$ and $T^n x$ are conjugate). 
\end{corollary}
Corollary \ref{main2} follows immediately from Theorem \ref{main1} applying the Criterion of local maximality in terms of Jacobi fields  (Theorem 1.1) of \cite{Bialy-Tsodikovich}. 

In order to deduce Corollary \ref{cor:vertical} one can argue analogously to \cite{B0}.
More precisely, suppose by contradiction, that for any vertical  vector $v\in T_x\mathcal B$ and any integer $n$, the vector $DT^n(v)$ is not vertical. This implies that any finite segment of a billiard trajectory $\{\gamma(s_n), n\in[M,N]\}$ has a non-degenerate matrix of second variation $\delta ^2\mathcal L_{MN}$.
Then, by a continuity argument, all the matrices $\delta ^2\mathcal L_{MN}$ must be negative definite (because this holds true for orbits lying on the rotational invariant curve $\alpha$). Hence all billiard  configurations, corresponding to the orbits lying in $\mathcal{B}$ are locally maximizing. Therefore, Theorem \ref{main1} applies and the curve $\gamma$ is a circle. Contradiction. 

\section*{Acknowledgements}
I am thankful to anonymous referees for careful reading 
and improving corrections.

\section{\bf Important tools}  
\subsection{Non-standard generating function}
Another way to get the same symplectic form is to fix an origin in $\mathbf R^2$ (we shall fix the origin at the center of $\gamma$) 
and to introduce the coordinates $(p,\varphi)$ on the space of all oriented lines, so that $\varphi$ is the angle between the right unit normal to the line and the horizontal and $p$ is the signed distance to the line (see Fig. \ref{fig:S}). In this way the space of oriented lines is identified with $T^*S^1$. Moreover,  the standard symplectic form  $d\beta$\ with $\beta=pd\varphi$
coincides with the symplectic form described before. In this description $p$ plays the role of momentum variable. 

For the second choice of the coordinates $(p,\varphi)$, the generating function was found first in \cite{BM} for the 2-dimensional case and then in \cite{B}  for higher dimensions (see \cite{BT} for further applications). This function $S$ is determined by the formulas:
$$
T^*\beta-\beta=p_1\,d \varphi_1-p\,d\varphi = dS, \quad S(\varphi ,\varphi _1 )=2h(\psi)\sin\delta,
$$
where $$\psi:=\frac{\varphi_1+\varphi}{2},\quad \delta:=\frac{\varphi_1-\varphi}{2}.$$
Here and  throughout this paper  we denote by $h$ the support function of $\gamma$ with respect to $0$:
$$
h(\psi):=\max_\gamma \left\langle\gamma,n_\psi \right\rangle,
$$
where $n_\psi$ is the unit outer normal to $\gamma$ in the direction $\psi$.
\begin{figure}[h]
	\centering
	\includegraphics[width=0.8\textwidth]{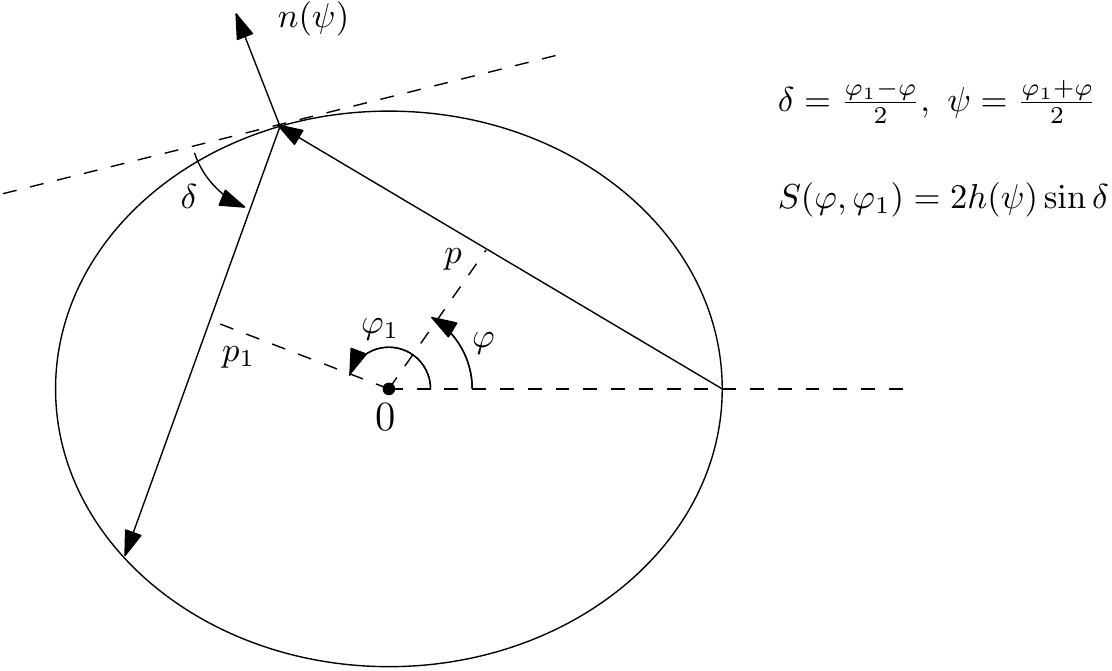}
	\caption{Generating function $S$ corresponding to the 1-form $\beta$}
	\label{fig:S}
\end{figure}
The fact that $S$ is the generating function for $T$ means that  the line with coordinates $(p,\varphi)$ is mapped into the line $(p_1,\varphi_1)$ (see Figure \ref{fig:S}) if and only if
\begin{equation}\label{generating}
	\begin{split}
		p&=-S_1(\varphi,\varphi_1)  =  h(\psi) \cos \delta-h'(\psi) \sin \delta,\\
		p_1&=S_2(\varphi,\varphi_1)  =  h(\psi) \cos \delta+h'(\psi) \sin \delta.
	\end{split}
\end{equation}

It follows from the direct computation (see below Proposition \ref{derivatives }) that the map $T$ satisfies the twist condition with respect to the symplectic coordinates $(p,\varphi)$ meaning that the cross-derivative satisfies $ S_{12}=\frac{1}{2}\rho(\psi)\sin\delta>0$, where $\rho(\psi)=h''(\psi)+h(\psi)>0$ is the radius of curvature.

\subsection{Two variational principles}
One can associate variational principle $\mathcal{S}$ also for the function $S$:
$$
\mathcal S\{\varphi_n\}=\sum_n S(\varphi_{n}, \varphi_{n+1}).
$$
In \cite{Bialy-Tsodikovich} we gave a criterion for an orbit to be locally maximizing. It then follows from this criterion that the set $\mathcal M$ does not depend on which generating function $L$ or $S$ is used for the map $T$.
We shall use the function $S$ in order to prove Theorem \ref{main1}.

{\bf Remark.}
	 It appears that vertical vector in the statement of the Corollary \ref{cor:vertical} can be understood with respect to each of the vertical foliations $\{s=const\}$ or $\{\varphi=const\}$. This follows from the proof of Corollary \ref{cor:vertical} and the fact, proven in \cite{Bialy-Tsodikovich}, that the classes of locally maximizing orbits corresponding to 
	 the generating functions $L, S$ coincide.
	
	In particular, the existence of conjugate points with respect to the vertical foliation $\{\varphi=const\}$ implies that one can find a beam of parallel lines such that after $n$ reflections the beam becomes parallel (infinitesimally) again. 
	 
\subsection{\bf Properties of the invariant curve of $4$-periodic orbits}\label{d-of-psi}If the billiard curve $\gamma$ is an ellipse, then
there exists a rotational invariant curve $\alpha$ consisting of $4$-periodic orbits. The corresponding quadrilaterals inscribed in $\gamma$ are called Poncelet $4$-gons.
It is  well-known (see \cite{Connes-Z} for several proofs) that all Poncelet 4-gons for an ellipse
are  parallelograms. This fact can be generalized  from the case of an  ellipse  
to any centrally-symmetric billiard table. We now turn to state the results from \cite{BM-annals} and refer to \cite{BM-annals} for the proofs. The next theorem is illustrated in Figure \ref{fig:parallelo}.

\begin{theorem}\label{4-periodic}
	Let $\gamma$ be a centrally-symmetric billiard table. Assume that billiard ball map $T:\mathbf A \rightarrow\mathbf A$ has a continuous rotational invariant curve $\alpha=\{\delta=d(\psi)\}$ of rotation number $\frac{1}{4}$ consisting of $4$-periodic orbits of $T$. Then the following properties hold:
	
	\begin{enumerate} [\rm (A)]
		\item Function $d(\psi)$ is  $\pi$-periodic and the billiard quadrilaterals
		corresponding to the traces of the orbits contained in the invariant curve
		$\alpha$ are parallelograms.
		
		\item The tangent lines to $\gamma$ at the vertices of the parallelogram form a rectangle.
		
		\item  $0<d(\psi)<\pi/2,\quad d\left(\psi+\frac{\pi}{2}\right)=\frac{\pi}{2}-d(\psi).$
		
		\item The functions $d$ and $h$ satisfy the identities  $$\tan d(\psi)=\frac{h(\psi)}{h\left(\psi+\frac{\pi}{2}\right)}=-\frac{h'\left(\psi+\frac{\pi}{2}\right)}{h'(\psi)}, \ $$
		and $$h^2(\psi)+h^2\left(\psi+\frac{\pi}{2}\right)=R^2=const.$$
	\end{enumerate}
\end{theorem}
\noindent{\bf Remark.}
	It follows from Theorem \ref{4-periodic} item (D), that the orthoptic curve  associated with  $\gamma$ is a circle of radius  $R$ (like in the case of an ellipse). Here the orthoptic curve of $\gamma$, by definition, is the locus of points $Q$, such that the two tangents to $\gamma$ passing through $Q$ form a right angle.

\begin{corollary}\label{relations1}
	Let $\gamma$ be a convex centrally-symmetric billiard table. Let $\alpha=\{\delta=d(\psi)\}\subset\mathbf A$ be an invariant curve consisting of $4$-periodic orbits. It then follows from Theorem \ref{4-periodic} item (D) that 
	$$
	h(\psi)= R\sin d(\psi) ,\quad h\left(\psi+\frac{\pi}{2}\right)=R\cos d(\psi), 
	$$
	for a positive constant $R$.
\end{corollary}
\begin{corollary}
	The explicit formulas of item (D) show that the invariant curve $\alpha$ is necessarily $C^2$-smooth, since the support function $h$ is $C^2$-smooth by assumption.
\end{corollary}
\begin{figure}[h]
	\centering
	\includegraphics[width=0.9\linewidth]{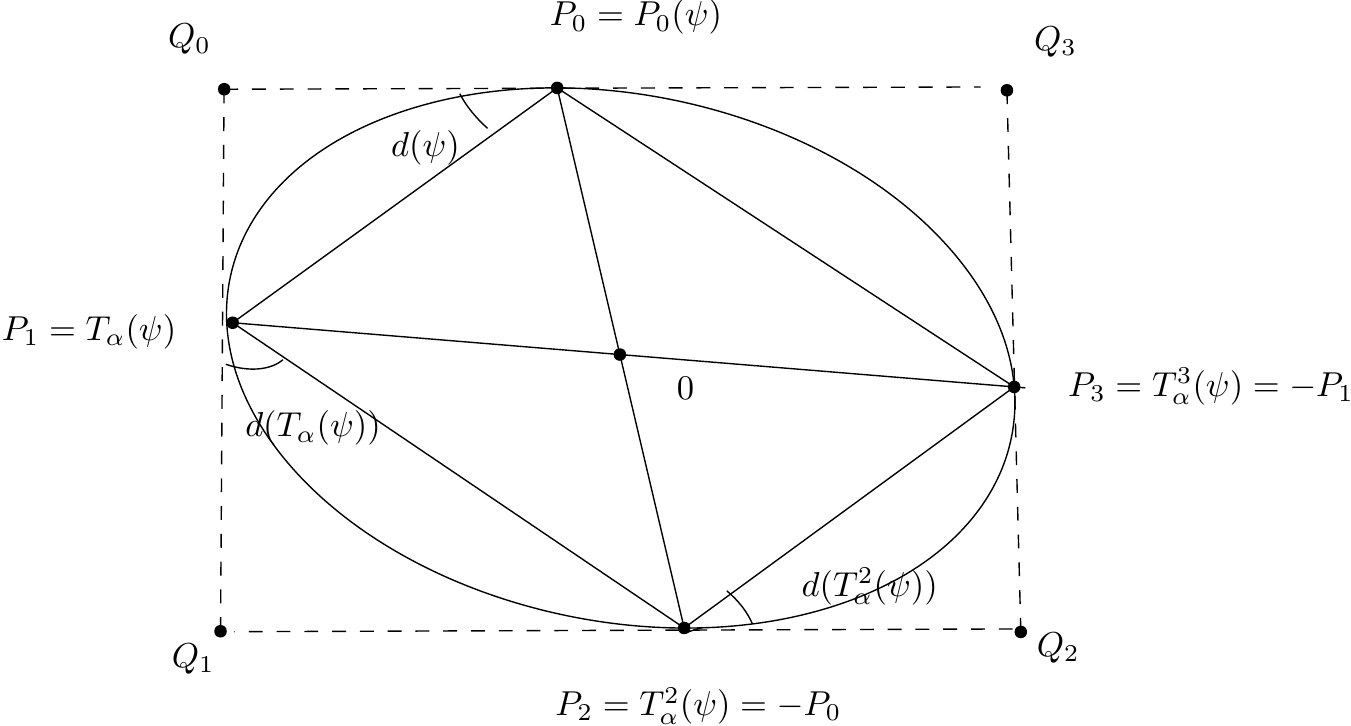}
	\caption{Rectangle $Q_0 Q_1 Q_2 Q_3$ corresponding to the $4$-periodic orbit forming parallelogram $P_0 P_1 P_2 P_3$. }
	\label{fig:parallelo}
\end{figure}
\subsection{Function $\omega$ and an inequality}
It turns out that one can introduce a measurable bounded function $\omega$ on the set $\mathcal M$ satisfying the inequality:
\begin{equation}\label{ineq:S}
	\begin{split}
		&\omega(p_1,\varphi_1)-\omega(p,\varphi)\geq\\
		\geq& S_{11}(\varphi,\varphi_1)+S_{22}(\varphi,\varphi_1) +2S_{12}(\varphi,\varphi_1).
	\end{split}
\end{equation}
The construction of this function (see \cite{B0}) was inspired by the celebrated E.Hopf theorem on tori with no conjugate points. 
Let me sketch this construction. 
Let $\{(p_n,\varphi_n)\}$ be a locally maximizing orbit of the point $z=(p_0,\varphi_0)$. It then follows that there exists an invariant vector field $\{(\delta p_n,\delta \varphi_n)\}$ along the orbit $\{(p_n,\varphi_n)\}$ such that the corresponding field $\delta \varphi_n$ is a Jacobi field along the billiard configuration $\{\varphi_n\}$  (normalized by $\delta \varphi_0=1$) and is strictly positive.  
Let me remind, a \textit{Jacobi field} along a configuration $\{\varphi_n\}$ is a sequence $\{\delta \varphi_n\}$ satisfying the
\textit{discrete Jacobi equation}:
\begin{equation}
	\label{eq:Jacobi}
	b_{n-1}\delta  \varphi_{n-1}+a_n
	\delta  \varphi_n+b_{n}\delta  \varphi_{n+1}=0,
\end{equation}
where, as before,
\[a_n=S_{22}( \varphi_{n-1}, \varphi_n)+S_{11}( \varphi_n, \varphi_{n+1}),\
b_n=S_{12}( \varphi_n, \varphi_{n+1}).\]

Then the invariance of the field $\{(\delta p_n,\delta\varphi_n)\}$ along the orbit implies (by differentiating the formula $p_n=-S_1(\varphi_n,\varphi_{n+1})$): 
\[
\delta p_n=-S_{11}(\varphi_n,\varphi_{n+1})\delta \varphi_n-S_{12}(\varphi_n,\varphi_{n+1})\delta \varphi_{n+1}
,\]
or equivalently, due to the Jacobi equation:
\[
\delta p_n=S_{22}(\varphi_{n-1},\varphi_{n})\delta \varphi_n+S_{12}(\varphi_{n-1},\varphi_{n})\delta \varphi_{n-1}.
\]
Then one defines $\omega(p_n,\varphi_n):=\frac{\delta p_n}{\delta \varphi_n}$. One can prove that $\omega$ is measurable function and satisfies the relations:
\begin{equation}\label{eq:omega-T}
	\begin{cases}
		\omega(T(p,\varphi))=S_{22}(\varphi,\varphi_1)+S_{12}(\varphi,\varphi_1)\delta \varphi_1(\varphi,p)^{-1},\\
		\omega(p, \varphi)=-S_{11}(\varphi, \varphi_1)-S_{12}(\varphi,\varphi_1)\delta \varphi_{1}(\varphi,p).
	\end{cases}
\end{equation}
Subtracting the second equation from the first one and using $S_{12}>0,$ $ \delta\varphi_1>0$ we get the inequality (\ref{ineq:S}).

Also notice, that  from the formulas  (\ref{eq:omega-T}) we have the inequality
$$
S_{22}(\varphi_{-1},\varphi)<\omega(p,\varphi)<-S_{11}(\varphi,\varphi_1),
$$ since  in (\ref{eq:omega-T})  $S_{12},\delta\varphi_1, \delta\varphi_{-1}$ are positive.
 Using  Proposition (\ref{derivatives }) it then follows, that function $\omega$ is bounded on $\mathcal M_\mathcal B$:
$$
|\omega|<\max_{\mathcal B} \{|S_{11}|, |S_{22}|\}<K(\gamma),
$$ where $K(\gamma)$ depends only on $\gamma$ (for example one can set $K(\gamma)=\max_\gamma\{\rho + h+ |h'|\}$, using formulas of Proposition (\ref{derivatives })).
\subsection{Derivatives of generating function $S$}
The derivatives of the generating function $S$ can be immediately computed: 
\begin{proposition}\label{derivatives }
	The  second partial derivatives of $S $ are:
	\begin{align*}
		&
		S_{11}(\varphi,\varphi_1)=\frac{1}{2}(h''(\psi)-h(\psi))\sin \delta -h'(\psi)\cos \delta;&\\
		&S_{22}(\varphi,\varphi_1)=\frac{1}{2}(h''(\psi)-h(\psi))\sin \delta +h'(\psi)\cos \delta;&\\
		&S_{12}(\varphi,\varphi_1)=\frac{1}{2}(h''(\psi)+h(\psi))\sin \delta ,&
	\end{align*}
	where $\psi:=\frac{\varphi_1+\varphi}{2},\quad \delta:=\frac{\varphi_1-\varphi}{2}.$
\end{proposition}

\section{\bf Proof of Theorem \ref{main1}}
In the sequel we shall work with the coordinates $(p,\varphi)$ and the function $\omega$
constructed above.
We start the proof of Theorem \ref{main1}
integrating (\ref{ineq:S}) over $\mathcal M_\mathcal B$ with respect to the invariant measure $d\mu=dpd\varphi$.

In order to perform the integration we compute the invariant measure as follows.

The symplectic form $dp \wedge d\varphi $  can be written using generating function (\ref{generating}):
$$
dp\wedge d\varphi=-d(S_1(\varphi,\varphi_1))\wedge d\varphi=S_{12}d\varphi\wedge d\varphi_1.
$$
Since $T$ is symplectic, the measure
$$d\mu= dp d\varphi=S_{12}d\varphi d\varphi_1,$$
is invariant.
Using the explicit formula for the second derivative (Proposition \ref{derivatives }) we compute:
$$
d\mu= S_{12}d\varphi d\varphi_1=
$$
$$
=\left(\frac{1}{2}\rho(\psi)\sin\delta\right)d\varphi d\varphi_1=\rho(\psi)\sin\delta d\psi d\delta,$$ where (see Figure \ref{fig:S} showing all the notations)
$$\rho(\psi)=h''(\psi)+h(\psi) $$ is the radius of curvature of $\gamma$, and

$$\quad  \psi:=\frac{\varphi_1+\varphi}{2},\quad \delta:=\frac{\varphi_1-\varphi}{2}.
$$

Hence, integrating the inequality (\ref{ineq:S}) with respect to the invariant measure $d\mu$
we obtain:
$$
0\geq \int_{\mathcal M_\mathcal B}[S_{11}(\varphi,\varphi_1)+2S_{12}(\varphi,\varphi_1)+S_{22}(\varphi,\varphi_1) ]d\mu.
$$
Moreover, we get from Proposition \ref{derivatives } after obvious simplifications
$$
S_{11}(\varphi,\varphi_1)+2S_{12}(\varphi,\varphi_1)+S_{22}(\varphi,\varphi_1) =2h''(\psi)\sin\delta.
$$

Thus we get  from (\ref{ineq:S}) the inequality:
\begin{equation}\label{eq:I-negative}
		0\geq  \int_{\mathcal M_\mathcal B}[(h''(\psi)\sin\delta]d\mu.
	\end{equation}
Since $\mathcal M_\mathcal B=\mathcal B\setminus\Delta_\mathcal B$ we get:
\begin{equation}\label{eq:Bdelta}
		\int_{\mathcal B}[(h''(\psi)\sin\delta]d\mu\leq\int_{\Delta_\mathcal B}[(h''(\psi)\sin\delta]d\mu
\end{equation}

Let us denote 
\begin{equation}\label{eq:I}
	I:=\int_{\mathcal B}[(h''(\psi)\sin\delta]d\mu.
\end{equation}
We shall give an upper bound for the right hand side of (\ref{eq:Bdelta}), and a lower bound on the left hand side $I$, and together we get the required bound.
For the right hand side of (\ref{eq:Bdelta}), write:
\begin{equation}
	\begin{split}
	&\int\limits_{\Delta_\mathcal B} h''(\psi)\sin\delta d\mu \leq \Big| \int\limits_{\Delta_\mathcal B} h''(\psi)\sin\delta d\mu \Big|\leq \\
		&\leq \int\limits_{\Delta_\mathcal B}|h''(\psi)\sin\delta| d\mu \leq \mu(\Delta_\mathcal B)\max\limits_{\Delta_\mathcal B} |h''|.
	\end{split}
\end{equation}
Since $h(\psi)+h''(\psi)=\rho(\psi)$, where $\rho(\psi)$ is the radius of curvature, then $$|h''|\leq\rho+\max h.$$
Since $\gamma$ is centrally symmetric, we have $\max h\leq D/2$, where $D$ is the diameter.
Also, the maximal radius of curvature of $\gamma$ is $\frac{1}{\beta}$ where $\beta$ is the minimal curvature of $\gamma$.
This gives us the estimate
\begin{equation}\label{eq:RHSEffectiveBound}I=\int_{\mathcal B}[(h''(\psi)\sin\delta]d\mu\leq
	\int\limits_{\Delta_\mathcal B} h''(\psi)\sin\delta d\mu \leq \Big(\frac{D}{2}+\frac{1}{\beta}\Big)\mu(\Delta_\mathcal B)\leq \frac{2}{\beta}\mu(\Delta_\mathcal B),
\end{equation}
where we used Blaschke's rolling disk theorem, stating that $\gamma$ is contained inside a disk with radius equal to the maximal radius of curvature of $\gamma$, and this means that $D\leq\frac{2}{\beta}$.

We now turn to estimate $I$ from below. Namely, we shall prove in the next Section the following:
\begin{theorem}\label{thm:I-below}Integral $I$ can be estimated from below:
$$I\geq \frac{3}{8}(P^2-4\pi A). $$
\end{theorem}
 Proof of Theorem \ref{main1} follows immediately from (\ref{eq:RHSEffectiveBound}) and Theorem \ref{thm:I-below}.
  
  \section{\bf Proof of Theorem \ref{thm:I-below} }
  Substituting into the integral $I$ the explicit expression $d\mu=\rho(\psi)\sin\delta d\psi d\delta$ and integrating first with respect to $\delta$, we get from (\ref{eq:I}):
\begin{equation}\label{eq-integral-delta}
		I= \int_{0}^{2\pi}d\psi\ \ \Big[ h''(h+h'')\int_{d(\psi)}^{\pi -d(\psi)}d\delta\ \sin^2\delta\Big].
	\end{equation}

Here we used the fact that  in the coordinates $(\psi,\delta)$ the domain of integration takes the form $$\mathcal B=\{(\psi,\delta):\psi\in[0,2\pi],\  \delta\in[d(\psi), \pi-d(\psi)]\}.$$Here and below $d(\psi)$ is the function described in Subsection \ref{d-of-psi}

Integrating in (\ref{eq-integral-delta}) with respect to  $\delta$ we obtain
\begin{equation}\label{krok}
		I= \int_{0}^{2\pi} \big[h''(\psi)(h''(\psi)+h(\psi))\big]\left(\frac{\pi}{2}-d(\psi)+\frac{1}{2}\sin 2d(\psi)\right)d\psi.
	\end{equation}
Now we substitute into (\ref{krok})  the expressions  for $h,h',h''$ via $d(\psi)$ using Corollary \ref*{relations1} of Theorem \ref{4-periodic}:
\begin{equation}\label{relations}
	\begin{cases}h=R\sin d,\\
		h'=R \cos d \ d', \\
		h''=R\cos d\  d''-R\sin d\ ( d')^2,\\
		d(\psi+\frac{\pi}{2})=\frac{\pi}{2}-d(\psi).
	\end{cases}
\end{equation} 
In what follows we usually omit the arguments for the functions $h,d$ and their derivatives. 

Thus we get from (\ref{eq:RHSEffectiveBound}) the following equality on the function $d$:
\begin{equation}\label{krokodil1}\begin{split}
		 I=& {R^2}\int_{0}^{2\pi} \left(\sin d-\sin d\  d'^2+\cos d\ d''\right)\left(-\sin d\  d'^2+\cos d\  d''\right)\\
		&\left(\frac{\pi}{2}-d+\frac{1}{2}\sin 2d\right)d\psi={R^2}\int_{0}^{2\pi}Ud\psi,
	\end{split}
\end{equation}
where we introduced $U$ by the formula $$
U:= \left(\sin d-\sin d\  d'^2+\cos d\ d''\right)\left(-\sin d\  d'^2+\cos d\  d''\right)\\
\left(\frac{\pi}{2}-d+\frac{1}{2}\sin 2d\right).
$$
The assumption of central symmetry implies
that $h(\psi),d(\psi)$ are $\pi$-periodic. Hence: 
$$
\int_0^{2\pi}U(\psi)d\psi=2 \int_0^{\pi}U(\psi)d\psi.
$$
We shall prove now the following
\begin{theorem}\label{thm:U-below}
	$$
	\int_0^{\pi}U(\psi)d\psi\geq \frac{3}{16R^2}(P^2-4\pi A).$$

\end{theorem}
\begin{proof}The idea of the proof is to proceed in three steps: "symmetrization", integration by parts, and Wirtinger inequality. Doing this, we pass to a new integrand, $\tilde U$, satisfying the inequality $\tilde U\geq const\  h'^2$. Moreover, this inequality enables one to estimate the integral of $\tilde U$ from below by isoperimetric defect. 
	
We write $$
U=U_1+U_2+U_3+U_4+U_5,
$$
where
\begin{align*}
&U_1=d''^2\cos^2 d\left(\frac{\pi}{2}-d+\frac{1}{2}\sin 2d\right),\\
&U_2=-2d''d'^2\sin d\cos d \left(\frac{\pi}{2}-d+\frac{1}{2}\sin 2d\right),\\
&U_3=d''\sin d\cos d \left(\frac{\pi}{2}-d+\frac{1}{2}\sin 2d\right),\\
&U_4=d'^4\sin^2 d\left(\frac{\pi}{2}-d+\frac{1}{2}\sin 2d\right),\\
&U_5=-d'^2\sin^2 d\left(\frac{\pi}{2}-d+\frac{1}{2}\sin 2d\right).
\end{align*}

{\underline{Step 1. Symmetrization.}}

We perform the change  of the integration variable by the rule $\psi\rightarrow\psi+\frac{\pi}{2}$.  By (\ref{relations}), which is the consequence of Theorem \ref{4-periodic} and Corollary \ref{relations1}, this intertwines $\sin (d)$ with $\cos(d)$ and changes the  sign of $d''$.  Denote the changed integrand by $\hat U_j$

Also denote the "symmetrized" integrand by 
$$
V_j:=U_j+\hat U_j.
$$ 
Then we have
$$
\int_0^{\pi}U_j(\psi)d\psi=\int_0^{\pi}\hat U_j(\psi)d\psi=\frac{1}{2}\int_0^{\pi}V_j(\psi) d\psi,
$$
where $V_j$ can be written as:

\begin{align*}
	&V_1=d''^2\left(\frac{\pi}{4}+(\frac{\pi}{4}-d)\cos 2d+\frac{1}{2}\sin 2d\right),\\
	&V_2=d''d'^2\sin 2d\left(2d-\frac{\pi}{2}\right),\\
	&V_3=d''\sin 2d\left(\frac{\pi}{4}-d\right),\\
	&V_4=d'^4\left(\frac{\pi}{4}+(d-\frac{\pi}{4})\cos 2d+\frac{1}{2}\sin 2d\right),\\
	&V_5=-d'^2\left(\frac{\pi}{4}+(d-\frac{\pi}{4})\cos 2d+\frac{1}{2}\sin 2d\right).
\end{align*}
	{\underline{Step 2. Integration by parts.}} 
We  apply  integration by parts for $V_2,V_3$ in order to get rid of the second derivative $d''$ (in the first power). Notice that thanks to the $\pi$-periodicity of the integrands, the  off-integration terms vanish. Thus we get new integrands $W_i,i=1,..,5$, where 
\begin{align*}
	&W_1=V_1=d''^2\left(\frac{\pi}{4}+(\frac{\pi}{4}-d)\cos 2d+\frac{1}{2}\sin 2d\right),\\
	&W_2=d'^4\left(-\frac{4}{3}\cos 2d\ (d-\frac{\pi}{4})-\frac{2}{3}\sin 2d\right),\\
	&W_3=d'^2\left(2\cos 2d\ (d-\frac{\pi}{4})+\sin 2d\right),\\
	&W_4=V_4=d'^4\left(\frac{\pi}{4}+(d-\frac{\pi}{4})\cos 2d+\frac{1}{2}\sin 2d\right),\\
	&W_5=V_5=-d'^2\left(\frac{\pi}{4}+(d-\frac{\pi}{4})\cos 2d+\frac{1}{2}\sin 2d\right).
\end{align*}
Thus we get for the integral of $U$:
\begin{equation}
	\int_{0}^{\pi}U d\psi=\int_{0}^{\pi}\sum_{i=1}^{5}U_id\psi=\frac{1}{2} \int_{0}^{\pi}\sum_{i=1}^{5}V_id\psi=\frac{1}{2} \int_{0}^{\pi}\sum_{i=1}^{5}W_id\psi.
\end{equation}
Summing $W_2+ W_4$ and $W_3+ W_5$ we rewrite using only three summands:
\begin{align*}
	&X_1:=W_1=d''^2\left(\frac{\pi}{4}+(\frac{\pi}{4}-d)\cos 2d+\frac{1}{2}\sin 2d\right),\\
	&X_2:=W_2+W_4=d'^4\left(\frac{\pi}{4}-\frac{1}{3}(d-\frac{\pi}{4})\cos 2d-\frac{1}{6}\sin 2d\right),\\
	&X_3:=W_3+W_5=d'^2\left(-\frac{\pi}{4}+(d-\frac{\pi}{4})\cos 2d+\frac{1}{2}\sin 2d\right).\\
\end{align*}
Thus we have:
\begin{equation}\label{eq:UX}
	\int_{0}^{\pi}U d\psi=\frac{1}{2} \int_{0}^{\pi}(X_1+X_2+X_3)d\psi.
\end{equation}
	{\underline{Step 3. Use of the Wirtinger inequality.}}
	 
	 Let us introduce the function of $d$ which is the multiplier in $X_1$:
	 $$
	 f(d):=\frac{\pi}{4}+\left(\frac{\pi}{4}-d\right)\cos 2d+\frac{1}{2}\sin 2d.
	 $$
	 This function is strictly positive since $d$ varies in $(0, \pi/2)$. In fact one can say more precisely $$f\in \left[\frac{1}{2}+\frac{\pi}{4},\frac{\pi}{2}\right).$$
	 Next we can write 
	 $$
	 X_2=d'^4 f_2,\quad f_2:=\frac{\pi}{4}-\frac{1}{3}\left(d-\frac{\pi}{4}\right)\cos 2d-\frac{1}{6}\sin 2d,
	 $$ and one can see that $f_2$ is positive as well. 
	 
	 Similarly for $X_3$ we have
	 $$
	 X_3= d'^2 f_3,\quad f_3:=-\frac{\pi}{4}+(d-\frac{\pi}{4})\cos 2d+\frac{1}{2}\sin 2d=(\sin 2d -f).
	 $$However the function $f_3$ is not necessarily positive.
	 In order to  bypass this difficulty, we shall use Wirtinger inequality, which we apply to the function
	 $$
	 Y:=d' \sqrt{f}.
	 $$
	 Notice, that $Y$ is $\pi$-periodic and has zero average, since it can be written as a complete derivative. Hence
	 $$
	 \int_{0}^{\pi}(Y'^2-4Y^2) d\psi\geq 0.
	 $$
	 We have the following expressions:
	 $$
	 Y'=\sqrt{f}d''+\frac{f'}{2\sqrt{f}}d'^2\Rightarrow Y'^2=f\ d''^2+f'\ d''d'^2+\frac{f'^2}{4f}d'^4.
	 $$
	 Therefore 
	 $$
	 \int_{0}^{\pi}(Y'^2-4Y^2) d\psi=\int_{0}^{\pi}(f\ d''^2+f'\ d''d'^2+\frac{f'^2}{4f}d'^4-4fd'^2)d\psi=
	 $$
	 $$
	 \int_{0}^{\pi}(f\ d''^2-\frac{f''}{3}d'^4+\frac{f'^2}{4f}d'^4-4fd'^2)d\psi:=\int_{0}^{\pi}gd\psi\geq 0,
	 $$where we performed integration by parts again.
	 
	 Thus finally we can write:
	 $$
	 X_1+X_2+X_3=g+d'^4\left(f_2+\frac{f''}{3}-\frac{f'^2}{4f}\right)+(f_3+4f)=$$
	 $$g+d'^4\left(f_2+\frac{f''}{3}-\frac{f'^2}{4f}\right)+d'^2(\sin 2d+3f).$$
	 
	 The following claim is crucial:
	 \begin{lemma}\label{lm:brackets}
Both expressions $(\sin2d+3f)$ and $\left(f_2+\frac{f''}{3}-\frac{f'^2}{4f}\right)$ of the last formula  are strictly positive.
	 \end{lemma}
 \begin{proof}
 
 1)	Since $f\in \left[\frac{1}{2}+\frac{\pi}{4},\frac{\pi}{2}\right),$ then $(3f+\sin2d)\geq \frac{3}{2}+\frac{3\pi}{4}$.
 Analyzing the behavior of the function $f$ one can claim more: 
 \begin{equation}\label{eq:f}
 	(3f+\sin 2d)\geq 3f(0)=\frac{3\pi}{2} .
 \end{equation}
 
 2) For the expression $\left(f_2+\frac{f''}{3}-\frac{f'^2}{4f}\right)$  we need to compute: $$
 f'=-2 \left(\frac{\pi}{4}-d\right)\sin 2d,
 $$
 $$
 f''=-4\left(\frac{\pi}{4}-d\right)\cos 2d+2\sin 2d.
 $$
 We substitute $f_2$ and the second derivative of $f$:
 $$
 \left(f_2+\frac{f''}{3}-\frac{f'^2}{4f}\right)=\frac{f''}{3}-\frac{f'^2}{4f}+
 \frac{\pi}{4}-\frac{1}{3}\left(d-\frac{\pi}{4}\right)\cos 2d-\frac{1}{6}\sin 2d=
 $$
 $$
 -\frac{f'^2}{4f}+\frac{1}{3}\left[-4\left(\frac{\pi}{4}-d\right)\cos 2d+2\sin 2d\right]+
 \frac{\pi}{4}-\frac{1}{3}\left(d-\frac{\pi}{4}\right)\cos 2d-\frac{1}{6}\sin 2d=
 $$
 $$
 -\frac{f'^2}{4f}+\frac{\pi}{4}+\left(d-\frac{\pi}{4}\right)\cos 2d+\frac{1}{2}\sin2d.
 $$
 Thus we need to check the sign of the expression:
 $$
 -\left(\frac{\pi}{4}-d\right)^2\sin^2 2d+\left(\frac{\pi}{4}+\left(\frac{\pi}{4}-d\right)\cos 2d+\frac{1}{2}\sin 2d\right)\left(\frac{\pi}{4}+\left(d-\frac{\pi}{4}\right)\cos 2d+\frac{1}{2}\sin2d\right)=
 $$
 $$
 -\left(d-\frac{\pi}{4}\right)^2+\frac{\pi^2}{4^2}+\frac{1}{4}\sin^2 2d+ \frac{\pi}{4}\sin 2d=$$
 $$-\left(d-\frac{\pi}{4}\right)^2+\left( \frac{\pi}{4}+\frac{1}{2}\sin 2d\right)^2.
 $$
 Notice that since $d\in(0, \frac{\pi}{2})$ then $\abs{d-\frac{\pi}{4}}<\frac{\pi}{4}$ and hence the last expression is strictly positive. This completes the proof of Lemma \ref{lm:brackets}.
 \end{proof}
We are now in position to finish the proof of Theorem \ref{thm:U-below}. Using Lemma \ref{lm:brackets} we can deduce from (\ref{eq:UX}) with the help of (\ref{eq:f})
\begin{equation}\label{eq:Uh}
	\begin{split}
		&2\int_{0}^{\pi}U d\psi\geq \int_{0}^{\pi}d'^2(\sin 2d+3f)d\psi\geq\int_{0}^{\pi} \frac{3\pi}{2} d'^2\ d\psi\geq\\
		& \int_{0}^{\pi}\frac{3\pi}{2} \cos^2d\  d'^2\  d\psi= \frac{3\pi}{2R^2} \int_{0}^{\pi} h'^2 d\psi,\\
	\end{split}
\end{equation}
where we used $h'=R\cos d\  d'$ of (\ref{relations}) in the last equality.

Now consider the isoperimetric defect $P^2-4\pi A$ for the curve $\gamma$.
We have the classical formulas:
$$P=\int_{0}^{2\pi} h d\psi,\quad A=\frac{1}{2}\int_{0}^{2\pi} (h^2-h'^2) d\psi.
$$ By Cauchy-Schwartz inequality we have: 
$$
P^2\leq 2\pi\int_{0}^{2\pi}h^2 d\psi=2\pi\left(2A+\int_{0}^{2\pi}h'^2 d\psi\right).
$$
Hence using (\ref{eq:Uh}) we get:
$$
P^2-4\pi A\leq 2\pi \int_{0}^{2\pi}h'^2 d\psi=4\pi \int_{0}^{\pi}h'^2 d\psi\leq \frac{16}{3}R^2\int_{0}^{\pi} U d\psi.
$$
This completes the proof of Theorem \ref{thm:U-below}.
\end{proof}
%%%%%%%%%%%%%%%%%%%%%%%%%%%%%
\section{Discussion}
It is very natural to ask if one can reconstruct elliptic billiards by sharp inequalities containing the measures $\mu(\Delta_\mathcal{B}),\mu(\mathcal{M}_\mathcal{B})$ (similarly to Theorem \ref{main1}).

It would be very interesting to extend the ideas used in this paper to other Hamiltonian systems such as Twist symplectic maps, as well as to continuous time systems.

An important goal in the study of Birkhoff billiards, as well as of general twist maps, in particular of standard-like maps, is to understand the dynamical behavior between two invariant curves. Our result can be considered as a step in this direction.
It is not clear, however, how to approach this goal for arbitrary invariant curves and also how to remove the central-symmetry assumption.

%%%%%%%%%%%%%%%%%%%%%%%%%%%%%%%%%%%%%%%


\begin{thebibliography}{90}
	
%{	\bibitem{Arnaud} M.-C. Arnaud. {\it Green bundles and related topics}.  Proceedings of the ICM. Volume III, 1653--1679, Hindustan Book Agency, New Delhi, 2010.}
	
%{	\bibitem{Arnaud1} M. Arcostanzo, M.-C. Arnaud, P. Bolle, M. Zavidovique. {\it Tonelli Hamiltonians without conjugate points and $ C^0$ integrability.} Math. Z. 280 (2015), no. 1--2, 165--194. }
	{\bibitem{AB} M. Arnold, M. Bialy. {\it Nonsmooth convex caustics for Birkhoff billiards.} Pacific J. Math. 295 (2018), no. 2, 257--269. }
	
	\bibitem{AKS} A. Avila, V. Kaloshin, J. De Simoi, {\it An integrable deformation of an ellipse of small eccentricity is an ellipse},  Ann. of Math., Vol. 184 (2016), 527--558.
	
	
	%
	
	
	
	%\bibitem{avila} Avila, Artur; De Simoi, Jacopo; Kaloshin, Vadim An integrable deformation of an ellipse of small eccentricity is an ellipse. Ann. of Math. (2) 184 (2016), no. 2, 527–558.
	{\bibitem{bangert} V. Bangert. {\it Mather set for twist maps and geodesics on tori}. Dynamics reported (1988), Vol. 1, 1-- 56, Dynam. Report. Ser. Dynam. Systems. Appl., 1, Wiley, Chichester.}
	\bibitem{B0} M. Bialy. {\it Convex billiards and a theorem by E. Hopf.} Math. Z. 214 (1993), no. 1, 147--154. 
	
	
	
	\bibitem{B1} M. Bialy. {\it Hopf rigidity for convex billiards on the hemisphere and hyperbolic plane}. Discrete Contin. Dyn. Syst. 33 (2013), no. 9, 3903--3913. 
	
	\bibitem {BM}  M. Bialy, A. Mironov. {\it Angular billiard and algebraic Birkhoff conjecture.} Adv. Math. { 313} (2017), 102--126.
	
	\bibitem{BM-annals}M. Bialy, A.E. Mironov. {\it The Birkhoff-Poritsky conjecture for centrally-symmetric billiard tables.} Ann. of Math. (2) 196 (2022), no. 1, 389–413.
	
%	\bibitem {BMT}  M. Bialy, A. Mironov, S. Tabachnikov. {\it Wire billiards, the first steps}, Adv. Math. 368 (2020), 107154, 27 pp.
	
	\bibitem{B}  M. Bialy. {\it Gutkin billiard tables in higher dimensions and rigidity.} Nonlinearity {\bf 31} (2018),  2281--2293.
	\bibitem{B-effective}M. Bialy. {\it Effective bounds in E. Hopf rigidity for billiards and geodesic flows.} Comment. Math. Helv. 90 (2015), no. 1, 139–153.
	\bibitem {BT} M. Bialy, S. Tabachnikov {\it Dan Reznik identities and more.} arXiv:2001.08469.
	\bibitem{Bialy-Tsodikovich}M. Bialy, D. Tsodikovich. {\it Locally maximizing orbits for the non-standard generating function of convex billiards and applications.} Nonlinearity {\bf 36} (2023), no. 3, 2001–2019.
	
%	\bibitem{bolotin} S.V. Bolotin. {\it Integrable Birkhoff billiards.} (Russian) Vestnik Moskov. Univ. Ser. I Mat. Mekh. (1990), no. 2, 33--36.
	
	\bibitem{Connes-Z} A. Connes, D. Zagier. {\it  A property of parallelograms inscribed in ellipses.}
	Amer. Math. Monthly {114} (2007),  909--914.
	
%	\bibitem{delshams} A. Delshams, R. Ramirez-Ros. {\it On Birkoff's [Birkhoff's] conjecture about convex billiards}. Proceedings of the 2nd Catalan Days on Applied Mathematics (Odeillo, 1995), 85–94, Collect. Études, Presses Univ. Perpignan, Perpignan, 1995. 
	
	%\bibitem {GS}  A. Glutsyuk, E. Shustin. {\it On polynomially integrable planar outer billiards and curves with symmetry property.} Math. Ann. 372 (2018), no. 3--4, 1481--1501.
	
	
	\bibitem{G} A. Glutsyuk. {\it On polynomially integrable Birkhoff billiards on surfaces of constant curvature}. Journal of the European Mathematical Society, J. Eur. Math. Soc. (JEMS) 23 (2021), no. 3, 995–1049. 
	
%	\bibitem {Innami}N. Innami. {\it Geometry of geodesics for convex billiards and circular billards.} Nihonkai Math. J.,13(1), (2002) 73--120.
	
	
	%\bibitem{GM} V. Guillemin, R. Melrose. {\it A cohomological invariant of discrete dynamical systems.} In E. B. Christoffel %(Aachen/Monschau, 1979), pp. 672--679. Birkh\"auser, Basel-Boston, Mass., 1981.
	
	\bibitem{K-S}V. Kaloshin, A. Sorrentino. {\it On the local Birkhoff Conjecture for convex billiards},  Ann. of Math., Vol. 188 (2018), 315--380.
	
	\bibitem{K-S1}V. Kaloshin, A. Sorrentino. {\it On the integrability of Birkhoff billiards.},
	Phil. Trans. R. Soc. A 376: 2017 0419.
	http://dx.doi.org/10.1098/rsta.2017.0419.
	
	\bibitem{KS2}G. Huang, V. Kaloshin, A. Sorrentino.
	{\it Nearly circular domains which are integrable close to the boundary are ellipses.},
	Geom. Funct. Anal. 28 (2018), no. 2, 334–392.
	
%	{\bibitem{katok}A. Katok. Billiard table as a mathematician’s playground Lecture delivered on March 10, 1999. http://www.personal.psu.edu/axk29/pub/Russian-bill.pdf.}
	\bibitem{KT} V.V.Kozlov, D.V.Treshchëv.  {\it Billiards. A genetic introduction to the dynamics of systems with impacts.} Translations of Mathematical Monographs, 89. American Mathematical Society, Providence, RI, 1991. viii+171 pp.
	%\bibitem{L-T} M. Levi, S.  Tabachnikov. {\it The Poncelet grid and billiards in ellipses}. Amer. Math. Monthly {\bf 114} %(2007), 895--908.
	
	%\bibitem{Poritsky} H. Poritsky. {\it The billiard ball problem on a table with a convex boundary--an illustrative dynamical problem.} Ann. of Math. (2) 51 (1950), 446--470. 
	
	%\bibitem{Sch} R. Schwartz. {\it The Poncelet grid.} Adv. Geom. {\bf 7} (2007), 157--175. 
	\bibitem{siburg} K.F. Siburg,  {\it The principle of least action in geometry and dynamics.} Lecture Notes in Mathematics, 1844. Springer-Verlag, Berlin, 2004. xii+128 pp.
	\bibitem{Tab} S. Tabachnikov. {\it Geometry and billiards.}  Amer. Math. Soc., Providence, RI, 2005.
	
%	\bibitem{T} S. Tabachnikov. {\it On algebraically integrable outer billiards.} Pacific J. Math. 235 (2008), no. 1, 89--92.
	
	
	\bibitem{W} M.P. Wojtkowski. {\it Two applications of Jacobi fields to the billiard ball problem.} J. Differential Geom. 40 (1994), no. 1, 155--164.
\end{thebibliography}
\end{document}